%% file: ex_article.tex
\documentclass[hidelinks,onefignum,onetabnum]{siamart250211}

\input{ex_shared}

\ifpdf
\hypersetup{
  pdftitle={Split-Merge Revisited: A Scalable Approach to Generalized Eigenvalue Problems},
  pdfauthor={Xiaozhi Liu and Yong Xia}
}
\fi

\begin{document}

\maketitle

\begin{abstract}
The generalized eigenvalue problem (GEP) serves as a cornerstone in a wide range of applications in numerical linear algebra and scientific computing. However, traditional approaches that aim to maximize the classical Rayleigh quotient often suffer from numerical instability and limited computational efficiency, especially in large-scale settings. In this work, we explore an alternative difference-based formulation of GEP by minimizing a structured quadratic polynomial objective, which enables the application of efficient first-order optimization methods. We establish global convergence guarantees for these methods without requiring line search, and further introduce a transform-domain perspective that reveals the intrinsic connection and performance gap between classical first-order algorithms and the power method. Based on this insight, we develop an accelerated preconditioned mirror descent algorithm, which allows for flexible preconditioner design and improved convergence behavior. Lastly, we extend the recently proposed Split-Merge algorithm to the general GEP setting, incorporating richer second-order information to further accelerate convergence. Empirical results on both synthetic and real-world datasets demonstrate that our proposed methods achieve significant improvements over existing baselines in terms of both computational efficiency and numerical stability.
\end{abstract}



\section{Introduction}

The generalized eigenvalue problem (GEP) \cite{gep_2011} plays a fundamental role in numerical linear algebra, scientific computing, and statistical learning.
It forms the mathematical foundation for several key applications, including canonical correlation analysis \cite{cca_1992}, linear discriminant analysis \cite{lda_1992}, sufficient dimension reduction \cite{sdr_2005}, and harmonic retrieval estimation \cite{mathew2002quasi}.
These applications aim to extract meaningful low-dimensional structures within high-dimensional data, thereby facilitating downstream tasks such as regression \cite{kakade2007multi}, classification \cite{karampatziakis2014discriminative}, and semantic embedding learning \cite{dhillon2011multi}.

Formally, given a symmetric matrix $\boldsymbol{A} \in \mathbb{R}^{n \times n}$ and a positive definite (PD) matrix $\boldsymbol{B} \in \mathbb{R}^{n \times n}$, the symmetric-definite GEP seeks scalars $\lambda_i$ and nonzero vectors $\boldsymbol{u}_i$ satisfying
$$
\boldsymbol{A}\boldsymbol{u}_i = \lambda_i \boldsymbol{B} \boldsymbol{u}_i,\quad i=1,2,\ldots,n,
$$
where $\lambda_i$ are the generalized eigenvalues and $\boldsymbol{u}_i$ are the corresponding generalized eigenvectors of the matrix pair $(\boldsymbol{A}, \boldsymbol{B})$. The eigenvalues are assumed to be sorted in descending order: $\lambda_1 \geq \lambda_2 \geq \ldots\geq \lambda_n$.

In many practical scenarios, it suffices to compute only the top-$k$ generalized eigenvectors ($k$-GEP), as those associated with smaller eigenvalues typically contribute negligible useful information. In this work, we focus on the basic computational primitive of $1$-GEP, which corresponds to solving the following Rayleigh quotient maximization problem:
\begin{equation}
	\max_{\boldsymbol{x}\in \mathbb{R}^{n}} \frac{\boldsymbol{x}^T\boldsymbol{A}\boldsymbol{x}}{\boldsymbol{x}^T\boldsymbol{B}\boldsymbol{x}},\ \text{s.t.}\  \boldsymbol{x}\neq \boldsymbol{0}.
	\label{eq:gep_quotient}
\end{equation}

\paragraph{Quotient-based methods}
Most existing methods for the GEP aim to maximize the Rayleigh quotient in \eqref{eq:gep_quotient}, forming what we term quotient-based methods.
Many classical techniques for the standard eigenvalue problem (EP)\footnote{The standard EP is a special case of \eqref{eq:gep_quotient} when $\boldsymbol{B} = \boldsymbol{I}$.} can be naturally extended to the GEP, with notable examples including the power method \cite{pm_gep_1992}, Rayleigh quotient iteration \cite{rqi_gep_2005}, the Lanczos method \cite{gep_2011}, and the Jacobi-Davidson method \cite{jdqz_1998}.
To circumvent explicit computing $\boldsymbol{B}^{-1}$, recent advances have employed approximate linear system solvers. For example, Ge et al. \cite{ge2016efficient} proposed the GenELinK algorithm based on an inexact block power method, and Allen-Zhu and Li \cite{allen2017doubly} developed LazyEV, a doubly accelerated method with gap-free theoretical guarantees. However, these methods often require careful tuning of numerous hyperparameters, which can hinder practical deployment.
An alternative strategy reformulates problem \eqref{eq:gep_quotient} as a Riemannian optimization problem on the generalized Stiefel manifold \cite{absil2008optimization}, allowing the use of modern Riemannian optimization algorithms \cite{Absil2024optimization,xu2020practical}.

\paragraph{Difference-based methods}
Unlike the above quotient-based methods, Auchmuty \cite{auchmuty1991globally} proposed some unconstrained variational principles for solving the $1$-GEP of the matrix pair $(\boldsymbol{A}, \boldsymbol{B})$.
Their approach focuses on a special case where the objective function is a quartic polynomial in the variable $\boldsymbol{x}$, enabling the development of descent-based algorithms and Newton-type methods \cite{mongeau2004computing}. However, these higher-order approaches suffer from two key limitations: (1) the lack of global convergence guarantees, and (2) the need to solve linear systems with varying-coefficient matrices at each iteration.

Recently, Liu and Xia \cite{liu2025difference} introduced a generalized family of difference-based methods for the standard EP, motivated by the following quadratic difference formulation originally studied in \cite{auchmuty1991globally}:
$$
\min_{\boldsymbol{x} \in \mathbb{R}^n} \boldsymbol{x}^T\boldsymbol{x} - \left(\boldsymbol{x}^T \boldsymbol{A} \boldsymbol{x}\right)^{\frac{1}{2}},
$$
where $\boldsymbol{A}$ is a positive semidefinite (PSD) matrix. Based on this formulation, they proposed an algorithm called Split-Merge, which achieves optimal performance within the majorization-minimization (MM) framework.

In this paper, we extend the {Split-Merge} algorithm \cite{liu2025difference} to the GEP with an arbitrary PD matrix $\boldsymbol{B}$. This generalization enables difference-based optimization techniques to be applied beyond the standard case $\boldsymbol{B} = \boldsymbol{I}$, thereby broadening both their practical applicability and theoretical foundation.

\noindent \textbf{Contributions}.
The main contributions of this work are summarized as follows:
\begin{itemize}
	\item We thoroughly exploit the structure of the difference-based formulation for the GEP and establish the convergence of first-order optimization methods without requiring line search.
	\item Within a unified MM framework, we reveal the intrinsic gap between classical first-order optimization methods and the power method. We bridge this gap via a transform-domain framework, and further propose an accelerated algorithm based on a preconditioned mirror descent approach, which offers enhanced flexibility through a tunable preconditioner.
	\item To leverage richer second-order information of the objective, we extend the recently proposed Split-Merge approach \cite{liu2025difference} to the GEP setting, enabling improved convergence behavior.
	\item We validate the effectiveness of the proposed methods on both synthetic and real-world datasets, demonstrating significant performance gains compared to existing benchmark methods.
\end{itemize}

\noindent \textbf{Organization}.
The remainder of this paper is structured as follows.
\Cref{sec:preliminary} explores the foundational properties of the difference-based formulation, which is characterized by a quadratic polynomial objective.
\Cref{sec:Transform} establishes a fundamental connection between first-order optimization methods and the classical power method, motivating the design of an accelerated approach based on preconditioned mirror descent.
In \cref{sec:sm}, we broaden the applicability of the Split-Merge strategy by extending it to the more general GEP setting.
\Cref{sec:simulation} evaluates the effectiveness of the proposed algorithms through comprehensive experiments on both synthetic and real-world datasets.
Finally, \cref{sec:conclusion} concludes the paper and outlines potential directions for future research.

\noindent \textbf{Notation}.
$\boldsymbol{A}$ denotes a matrix, $\boldsymbol{a}$ a vector, and $a$ a scalar.
$\boldsymbol{A}^T$ and $\boldsymbol{A}^{-1}$ represent the transpose and inverse of $\boldsymbol{A}$, respectively.
$\boldsymbol{A} \succ 0$ indicates that $\boldsymbol{A}$ is PD, and $\boldsymbol{A} \succeq 0$ indicates that $\boldsymbol{A}$ is PSD.
$\|\boldsymbol{a}\|$ denotes the $\ell_2$-norm of $\boldsymbol{a}$.
$\operatorname{diag}(\boldsymbol{a})$ represents the diagonal matrix with the elements of $\boldsymbol{a}$ on its diagonal.

\section{Preliminaries}\label{sec:preliminary}

In this work, we address the $1$-GEP of the matrix pair $(\boldsymbol{A}, \boldsymbol{B})$ by solving the following unconstrained optimization problem:
\begin{equation}
	\min_{\boldsymbol{x} \in \mathbb{R}^n} \boldsymbol{x}^T\boldsymbol{B}\boldsymbol{x} - \left(\boldsymbol{x}^T \boldsymbol{A} \boldsymbol{x}\right)^{\frac{1}{2}},
	\label{eq:gep_difference}
\end{equation}
where $\boldsymbol{A}$ is symmetric PSD and $\boldsymbol{B}$ is symmetric PD.

\begin{remark}
	The assumption that $\boldsymbol{A}$ is PSD is without loss of generality, as we can shift $\boldsymbol{A}$ by adding $\eta \boldsymbol{B}$ for a sufficiently large $\eta>0$, ensuring $\boldsymbol{A} + \eta \boldsymbol{B}$ is PSD.
\end{remark}

\begin{remark}
	For the general $k$-GEP, the solution can be obtained by iteratively performing the $1$-GEP $k$ times using deflation techniques \cite{allen2017doubly}. Details are provided in \cref{appendix:deflation}.
\end{remark}

Let $f(\boldsymbol{x})$ denote the objective function of problem \cref{eq:gep_difference}. Using elementary calculus, we derive its gradient and Hessian as
\begin{equation}
	\nabla f\left(\boldsymbol{x}\right) = 2 \boldsymbol{B}\boldsymbol{x} - \frac{\boldsymbol{Ax}}{\left(\boldsymbol{x}^T\boldsymbol{A}\boldsymbol{x}\right)^{\frac{1}{2}}},
	\label{eq:grad_f}
\end{equation}
and
\begin{equation}
	\nabla^2 f\left(\boldsymbol{x}\right) = 2 \boldsymbol{B} - \frac{\boldsymbol{A}}{\left(\boldsymbol{x}^T\boldsymbol{A}\boldsymbol{x}\right)^{\frac{1}{2}}} + \frac{\left(\boldsymbol{Ax}\right) \left(\boldsymbol{Ax}\right)^T}{\left(\boldsymbol{x}^T\boldsymbol{A}\boldsymbol{x}\right)^{\frac{3}{2}}},
	\label{eq:hesse_f}
\end{equation}
respectively.

\begin{remark}
	The quadratic polynomial function $f(\boldsymbol{x})$ is differentiable on the set $\Theta = \left\{ \boldsymbol{x}: \boldsymbol{A}\boldsymbol{x} \neq \boldsymbol{0} \right\}$.
	While this differentiability has led previous works \cite{auchmuty1991globally,mongeau2004computing} to adopt its smooth quartic polynomial counterpart, such approaches fail to fully exploit the intrinsic structure of the objective function.
	In contrast, the function $f(\boldsymbol{x})$ exhibits several distinctive and computationally favorable properties that its smooth counterpart fails to preserve. These intrinsic features, which we will rigorously analyze in subsequent sections, offer significant benefits for both theoretical understanding and algorithmic design.
\end{remark}

Similar to the analysis in \cite{auchmuty1991globally}, we present the optimality characterization of problem \cref{eq:gep_difference}:
\begin{theorem} \label{thm:optimality}
	(i) All stationary points of the function $f\left(\boldsymbol{x}\right)$ are eigenvectors of the matrix pair $(\boldsymbol{A}, \boldsymbol{B})$. Moreover, the eigenvalue associated with any stationary point $\boldsymbol{x}$ is given by $\lambda\left(\boldsymbol{x}\right) = 2 \left(\boldsymbol{x}^T \boldsymbol{A} \boldsymbol{x}\right)^{\frac{1}{2}}$.
	
	(ii) The global minimizers of the optimization problem in \cref{eq:gep_difference} are the eigenvectors corresponding to the largest eigenvalue $\lambda_1$, and the corresponding minimum value is $-\frac{\lambda_1}{4}$.
	
	(iii) All second-order stationary points of the optimization problem in \cref{eq:gep_difference} are global minima. In particular, every local minimum is also a global minimum. Equivalently, all eigenvectors of the matrix pair $(\boldsymbol{A}, \boldsymbol{B})$ other than those associated with the dominant eigenvalue $\lambda_1$ are strict saddle points.
\end{theorem}

One notable advantage of the function $f(\boldsymbol{x})$, in contrast to its smooth quartic polynomial counterpart, is the existence of a positive Lipschitz constant \cite{extending2020}. Specifically, there exists a constant $L_{+} > 0$ such that  
\begin{equation}
	\max_{1\leq j \leq n} \max\left(\lambda_j\left(\boldsymbol{x}\right), 0\right) \leq L_{+},\ \forall \boldsymbol{x},
	\label{eq:pos_lip}
\end{equation}
where $\lambda_j\left(\boldsymbol{x}\right)$ denotes the $j$-th eigenvalue of the Hessian $\nabla^2 f\left(\boldsymbol{x}\right)$.

This property is critical for establishing the convergence of first-order optimization methods without line search.

Notably, condition \cref{eq:pos_lip} is equivalent to requiring that there exists a constant $L_{+} > 0$ such that
\begin{equation}
	\nabla^2 f(\boldsymbol{x}) \preceq L_+ \boldsymbol{I},\ \forall \boldsymbol{x},
	\label{eq:hessian_pos_lip}
\end{equation}
which bounds only the positive curvature of the function. This is strictly weaker than the standard Lipschitz gradient assumption, which requires a two-sided bound:
$$
-L \boldsymbol{I} \preceq \nabla^2 f(\boldsymbol{x}) \preceq L \boldsymbol{I},\ \forall \boldsymbol{x}.
$$
This distinction is significant because the positive curvature constant $L_+$ of our objective function $f(\boldsymbol{x})$ is much easier to estimate or compute, whereas general nonconvex functions may not possess such a structure.

\begin{remark}
	Although prior studies \cite{auchmuty1991globally,mongeau2004computing} have employed descent methods (e.g., steepest descent and conjugate gradient methods) to optimize quartic polynomial objective functions and have established global convergence guarantees, the lack of a positive Lipschitz constant for such functions necessitates the use of exact line search strategies. These line-search procedures incur additional computational overhead, which is often undesirable in practice.
\end{remark}

In the following, we identify a fundamental gap between first-order optimization methods and the classical power method for solving GEP. To bridge this gap, we propose a transform-domain framework that enables a principled acceleration of the power method.

\section{Transform-Domain Framework with Preconditioning} \label{sec:Transform}

\subsection{Motivation: Bridging First-Order Optimization and the Power Method} \label{sec:motivation_pmd}

The gradient descent (GD) method for solving problem \cref{eq:gep_difference} follows the standard update rule:
\begin{equation}
	\boldsymbol{x}_{k+1} = \boldsymbol{x}_{k} - \alpha \nabla f(\boldsymbol{x}_k), \label{eq:gd_gep}
\end{equation}
where $\alpha$ denotes the stepsize and $k$ is the iteration index.

When the stepsize satisfies the positive Lipschitz condition:
\begin{equation}
	\alpha L_{+} \in (0,2),
	\label{eq:positive_lip}
\end{equation}
one can establish the classical descent property of GD methods. This result follows from the fact that the objective function has bounded positive curvature, as formalized in the following lemma:
\begin{lemma} \label{lemma:positive_lipschitz}
	Let $f:\mathbb{R}^n\rightarrow\mathbb{R}$ be continuously differentiable, and suppose there exists a constant $L_+ > 0$ such that condition \cref{eq:hessian_pos_lip} holds. Then, for all $\boldsymbol{x},\boldsymbol{y}\in\mathbb{R}^n$, the following inequality is satisfied:
	\begin{equation}
		f(\boldsymbol{y}) \leq f(\boldsymbol{x}) + \nabla f(\boldsymbol{x})^T (\boldsymbol{y} - \boldsymbol{x}) + \frac{L_+}{2} \|\boldsymbol{y} - \boldsymbol{x}\|^2.
		\label{eq:upper_lip}
	\end{equation}
\end{lemma}

\begin{proof}
	See \cref{appendix:proof_positive_lipschitz}.
\end{proof}

By applying the inequality \cref{eq:upper_lip}, we naturally arrive at the following result.
\begin{lemma} [Descent Lemma with Bounded Positive Curvature] \label{lemma:descent}
	Suppose $f:\mathbb{R}^n\rightarrow\mathbb{R}$ is continuously differentiable, and its Hessian satisfies the positive curvature bound given in \cref{eq:hessian_pos_lip} for some constant $L_+ > 0$. Then, for any stepsize $\alpha \in (0, 2/L_+)$, the GD update $\bar{\boldsymbol{x}} = \boldsymbol{x} - \alpha \nabla f(\boldsymbol{x})$ ensures a sufficient decrease in the objective:
	$$
	f(\bar{\boldsymbol{x}}) \leq f(\boldsymbol{x}) - \alpha \left(1 - \frac{\alpha L_+}{2}\right) \|\nabla f(\boldsymbol{x})\|^2.
	$$
	In particular, the function value is non-increasing at each iteration: $f(\bar{\boldsymbol{x}}) \leq f(\boldsymbol{x})$.
\end{lemma}

\begin{proof}
	See \cref{appendix:proof_descent}.
\end{proof} 

Since the objective function $f$ in \cref{eq:gep_difference} is coercive, it is therefore bounded below. By leveraging \cref{lemma:descent}, we can further conclude that the sequence $\left\{\boldsymbol{x}_k\right\}$ generated by the GD scheme with a stepsize satisfying condition \cref{eq:positive_lip} converges to a stationary point of problem \cref{eq:gep_difference}; see the following theorem for a formal statement.

\begin{theorem} [Convergence to Stationary Points] \label{thm:convergence_stationary_point}
	let $f:\mathbb{R}^n\rightarrow\mathbb{R}$ be a continuously differentiable function whose Hessian satisfies the positive curvature bound given in \cref{eq:hessian_pos_lip} for some constant $L_+ > 0$. Suppose further that $f$ is bounded below. Let the sequence $\left\{\boldsymbol{x}_k\right\}$ be generated by the GD iteration:
	$$
	\boldsymbol{x}_{k+1} = \boldsymbol{x}_k - \alpha \nabla f(\boldsymbol{x}_k),
	$$
	with a fixed stepsize $\alpha \in (0, 2/L_+)$. Then the sequence of gradient norms converges to zero:
	$$
	\lim_{k\rightarrow \infty} \|f(\boldsymbol{x}_k)\| = 0.
	$$
\end{theorem}

\begin{proof}
	See \cref{appendix:proof_convergence_stationary_point}.
\end{proof} 

A recent result in \cite{extending2020} demonstrates that when the stepsize satisfies the positive Lipschitz condition \cref{eq:positive_lip}, the GD algorithm also avoids convergence to strict saddle points, as established in the following theorem:
\begin{theorem} [Non-Convergence to Strict Saddle Points] \label{thm:gd_convergence}
	Let $f \in C^2(\Omega)$, where $\Omega$ is a forward invariant convex subset of $\mathbb{R}^n$, and suppose that the gradient of $f$ has a positive Lipschitz constant $L_{+}$. Let $\sigma \left(\cdot\right)$ denote the spectrum of a matrix. Consider the GD update $\bar{\boldsymbol{x}} = \boldsymbol{x} - \alpha \nabla f(\boldsymbol{x})$ with $\alpha L_{+} \in (0,2)$, and assume that the set 
	$$\left\{\boldsymbol{x} \in \Omega \mid \alpha^{-1} \in \sigma\left(\nabla^2 f(\boldsymbol{x})\right)\right\}$$	
	has measure zero and contains no saddle points. Then, for a uniformly random initialization in $\Omega$, the probability of GD converging to a strict saddle point is zero.
\end{theorem}

Leveraging the fact that the objective function $f$ in \cref{eq:gep_difference} admits no non-strict saddle points (\cref{thm:optimality}(iii)), and invoking \cref{thm:convergence_stationary_point} and \cref{thm:gd_convergence}, we deduce that the GD method applied to \cref{eq:gep_difference} converges to a global minimizer of $f$ with probability one, under the stepsize condition specified in \cref{eq:positive_lip}.

\begin{theorem} [Convergence to Global Minimizer] \label{thm:global_convergence}
	Consider the objective function $f$ in problem \cref{eq:gep_difference}, and let $L_+ > 0$ denote its positive Lipschitz constant.
	Suppose the GD method is initialized at a point $\boldsymbol{x}_0 \in \mathbb{R}^n$ drawn uniformly at random, and uses a fixed stepsize $\alpha \equiv \alpha_0$, where $\alpha_0$ is sampled uniformly from the interval $(0,2/L_+)$. Then, the generated sequence $\left\{\boldsymbol{x}_k\right\}$ converges to a global minimizer of $f$ with probability one.
\end{theorem}

\begin{remark}
	(i) Based on the Hessian structure in \cref{eq:hesse_f}, the constant $L_{+}$ can be set to $2\lambda_1(\boldsymbol{B})$, where $\lambda_1(\boldsymbol{B})$ denotes the largest eigenvalue of the matrix $\boldsymbol{B}$. This constant can be efficiently estimated via standard eigensolvers. Alternatively, one may approximate $L_{+}$ by the trace of $\boldsymbol{B}$, though this may lead to slower convergence.
	
	(ii) In practice, the stepsize $\alpha$ is sampled uniformly at random from the interval $[0.9\times2/L_+, 0.99\times2/L_+]$ to promote faster convergence.
	
	(iii) Unlike prior first-order methods \cite{auchmuty1991globally,mongeau2004computing}, our approach does not require additional line-search procedures. A fixed stepsize suffices, as long as it satisfies the condition in \cref{eq:positive_lip}. This is a key motivation for employing the objective function $f(\boldsymbol{x})$ in \cref{eq:gep_difference} rather than its smooth quartic counterpart.
\end{remark}

We can interpret first-order optimization algorithms from a more general MM perspective. Specifically, the GD update in \cref{eq:gd_gep} can be viewed as the solution to the following surrogate minimization problem:
\begin{equation}
	\boldsymbol{x}_{k+1} = \arg \min_{\boldsymbol{x} \in \mathbb{R}^n} f(\boldsymbol{x}_k) + \left\langle \nabla f(\boldsymbol{x}_k), \boldsymbol{x} - \boldsymbol{x}_k \right\rangle + \frac{1}{2\alpha}\left\| \boldsymbol{x} - \boldsymbol{x}_k  \right\|^2,
	\label{eq:gd_mm}
\end{equation}
where $\alpha$ satisfies the condition given in \cref{eq:positive_lip}. 
In particular, when $\alpha \in \left(0, {1}/{2\lambda_1(\boldsymbol{B})}\right)$, the objective in \cref{eq:gd_mm} serves as a global surrogate for the original function $f(\boldsymbol{x})$, because the following inequality holds:
$$
\frac{1}{\alpha} \boldsymbol{I} \succeq \nabla^2 f(\boldsymbol{x}),\ \forall \boldsymbol{x}.
$$

\begin{remark}
	The update in \cref{eq:gd_mm} only incorporates information about the dominant eigenvalue of the matrix $\boldsymbol{B}$. This partial spectral usage contributes to the slow convergence behavior of standard GD.
\end{remark}
To exploit the full spectral information of $\boldsymbol{B}$, consider the alternative update:
\begin{equation}
	\boldsymbol{x}_{k+1} = \arg \min_{\boldsymbol{x} \in \mathbb{R}^n} f_k(\boldsymbol{x}),
	\label{eq:pm_sur}
\end{equation}
where
$$
f_k(\boldsymbol{x}) = f(\boldsymbol{x}_k) + \left\langle \nabla f(\boldsymbol{x}_k), \boldsymbol{x} - \boldsymbol{x}_k \right\rangle + \left(\boldsymbol{x} - \boldsymbol{x}_k\right) \boldsymbol{B} \left(\boldsymbol{x} - \boldsymbol{x}_k\right)
$$
is a global quadratic surrogate function of $f(\boldsymbol{x})$ at $\boldsymbol{x}_k$. Notably, its Hessian satisfies:
$$
2\lambda_1(\boldsymbol{B}) \boldsymbol{I} \succeq \nabla^2 f_k(\boldsymbol{x}_k) = 2 \boldsymbol{B} \succeq \nabla^2 f(\boldsymbol{x}_k).
$$
\begin{remark}
	The solution to \cref{eq:pm_sur} corresponds to solving the following linear system:
	$$
	\boldsymbol{B}\boldsymbol{x}_{k+1} = \frac{\boldsymbol{Ax}_k}{2(\boldsymbol{x}_k^T \boldsymbol{A} \boldsymbol{x}_k)^{\frac{1}{2}}}.
	$$
	This iteration is equivalent to the power method for the GEP \cite{pm_gep_1992}, ignoring normalization.
\end{remark}

While GD benefits from a simple iterative structure, its performance relies on knowledge of the dominant eigenvalue of matrix $\boldsymbol{B}$ and it fails to exploit richer second-order information. In contrast, the power method takes advantage of the full spectral structure of $\boldsymbol{B}$, but at the cost of solving linear systems and only partially utilizing the Hessian of $f(\boldsymbol{x})$. These observations naturally raise the question: Can we design an approach that effectively bridges the gap between first-order optimization and the power method, balancing simplicity with richer curvature information? This question serves as the key motivation for our proposed transform-domain framework.

\subsection{Preconditioned Mirror Descent Approach} \label{sec:pmd}

When applying GD to solve the GEP, two major limitations arise:
\begin{enumerate}
	\item Compared to the power method, GD attempts to approximate the full spectral structure of the matrix $\boldsymbol{B}$ using only information about its dominant eigenvalue. This approximation is reasonable only when the spectrum of $\boldsymbol{B}$ is dense, i.e., when the condition number $\kappa_{\boldsymbol{B}} \approx 1$.
	\item The convergence of GD relies on knowledge of the dominant eigenvalue of $\boldsymbol{B}$, which is generally unavailable in practice.
\end{enumerate}
To mitigate these issues, we propose performing GD in a transformed domain using a suitable preconditioner $\boldsymbol{P}$.
Specifically, the goal is to find a preconditioner such that the transformed matrix $\boldsymbol{\tilde{B}} = \boldsymbol{P}^{-T} \boldsymbol{B} \boldsymbol{P}^{-1}$ has a lower condition number and a dominant eigenvalue that is easier to estimate.
We then perform GD in the transformed domain $\boldsymbol{y} = \boldsymbol{P}\boldsymbol{x}$, solving the following equivalent optimization problem:
\begin{equation}
	\min_{\boldsymbol{y} \in \mathbb{R}^n} \boldsymbol{y}^T\boldsymbol{\tilde{B}}\boldsymbol{y} - \left(\boldsymbol{y}^T \boldsymbol{\tilde{A}} \boldsymbol{y}\right)^{\frac{1}{2}},
	\label{eq:gep_difference_transform}
\end{equation}
where $\boldsymbol{\tilde{A}} = \boldsymbol{P}^{-T} \boldsymbol{A} \boldsymbol{P}^{-1}$.

The corresponding alternating iterative scheme is given by:
\begin{equation}
	\begin{aligned}
		& \boldsymbol{y}_{k} = \boldsymbol{P} \boldsymbol{x}_{k}, \\
		& \boldsymbol{y}_{k+1} = \boldsymbol{y}_{k} - \alpha \left(2 \boldsymbol{\tilde{B}}\boldsymbol{y}_{k} - \frac{\boldsymbol{\tilde{A}}\boldsymbol{y}_k}{\left(\boldsymbol{y}_k^T\boldsymbol{\tilde{A}}\boldsymbol{y}_k\right)^{\frac{1}{2}}}\right), \\
		& \boldsymbol{x}_{k+1} = \boldsymbol{P}^{-1} \boldsymbol{y}_{k+1}.
	\end{aligned}
	\label{eq:pmd}
\end{equation}
Here, the stepsize $\alpha$ only needs to satisfy the condition $\alpha \in \left(0, {1}/{\lambda_1(\boldsymbol{\tilde{B}})}\right)$ to ensure convergence, as guaranteed by \cref{thm:global_convergence}.

An interesting observation is that the iterative scheme in \cref{eq:pmd} is equivalent to a mirror descent strategy \cite{bubeck2015convex} in the primal space, which also motivates the naming of our approach as preconditioned mirror descent (PMD).

\begin{theorem}
	Given a preconditioner $\boldsymbol{P}$, the iterative scheme in \cref{eq:pmd} can be interpreted as a mirror descent strategy with the mirror map $\Phi(\boldsymbol{x}) = \frac{1}{2} \boldsymbol{x}^T \boldsymbol{P}^T \boldsymbol{P} \boldsymbol{x}$:
	\begin{equation}
		\boldsymbol{x}_{k+1} = \arg\min_{\boldsymbol{x} \in \mathbb{R}^n} \alpha \nabla f(\boldsymbol{x}_k)^T \boldsymbol{x} + D_{\Phi}(\boldsymbol{x},\boldsymbol{x}_k),
		\label{eq:thm_pmd}
	\end{equation}
	where $D_{\Phi}(\boldsymbol{x}, \boldsymbol{y})$ denotes the Bregman divergence associated with $\Phi$, defined by
	$$
	D_{\Phi}(\boldsymbol{x}, \boldsymbol{y}) = \Phi(\boldsymbol{x})-\Phi(\boldsymbol{y})-\nabla \Phi(\boldsymbol{y})^T(\boldsymbol{x}-\boldsymbol{y}).
	$$
\end{theorem}
\begin{proof}
	It is straightforward to verify that the Bregman divergence induced by the mirror map $\Phi(\boldsymbol{x}) = \frac{1}{2} \boldsymbol{x}^T \boldsymbol{P}^T \boldsymbol{P} \boldsymbol{x}$ simplifies to the Mahalanobis distance:
	\begin{equation}
		D_{\Phi}(\boldsymbol{x}, \boldsymbol{y}) = \frac{1}{2} (\boldsymbol{x} - \boldsymbol{y})^T  \boldsymbol{P}^T \boldsymbol{P} (\boldsymbol{x} - \boldsymbol{y}).
		\label{eq:Mahalanobis}
	\end{equation}
	Now consider the update rule in \cref{eq:pmd}. Substituting its expression yields:
	$$
	\begin{aligned}
		\boldsymbol{x}_{k+1} & = \boldsymbol{P}^{-1}\left(\boldsymbol{P}\boldsymbol{x}_k - \alpha \left(2 \boldsymbol{P}^{-T}\boldsymbol{B}\boldsymbol{x}_{k} - \frac{\boldsymbol{P}^{-T}\boldsymbol{A}\boldsymbol{x}_k}{\left(\boldsymbol{x}_k^T\boldsymbol{A}\boldsymbol{x}_k\right)^{\frac{1}{2}}}\right)\right) \\
		& = \boldsymbol{x}_k - \alpha \left(\boldsymbol{P}^{T}\boldsymbol{P}\right)^{-1} \nabla f(\boldsymbol{x}_k).
	\end{aligned}
	$$
	This update is equivalent to:
	$$
	\boldsymbol{x}_{k+1} = \arg \min_{\boldsymbol{x} \in \mathbb{R}^n} \alpha \nabla f(\boldsymbol{x}_k)^T \boldsymbol{x} + \frac{1}{2}(\boldsymbol{x} - \boldsymbol{x}_k)^T \boldsymbol{P}^{T}\boldsymbol{P} (\boldsymbol{x} - \boldsymbol{x}_k).
	$$
	Substituting the expression for the Bregman divergence from \cref{eq:Mahalanobis} yields \cref{eq:thm_pmd}, completing the proof.
\end{proof}

Ideally, we aim to construct an optimal preconditioner $\boldsymbol{P}$ such that $\boldsymbol{\tilde{B}} = \boldsymbol{I}$. In this case, the condition number satisfies $\kappa_{\boldsymbol{\tilde{B}}} = 1$, and the dominant eigenvalue is exactly 1. The optimal preconditioner\footnote{The optimal preconditioner $\boldsymbol{P}$ is not unique. For example, one may also choose $\boldsymbol{P} = \boldsymbol{B}^{\frac{1}{2}}$, which, unlike the Cholesky factor, does not have a lower triangular structure.} $\boldsymbol{P}$ can then be obtained via the Cholesky decomposition $\boldsymbol{B} = \boldsymbol{L}\boldsymbol{L}^T$, by setting $\boldsymbol{P} = \boldsymbol{L}^T$.

As demonstrated in \cite{liu2025difference}, when $\boldsymbol{B} = \boldsymbol{I}$, GD with a fixed stepsize of $1/2$ becomes equivalent to the power method. In this sense, it provides an exact theoretical connection that fills the gap between GD and the power method. This property also holds in the GEP, as shown in the following theorem:

\begin{theorem}
	Let $\boldsymbol{P}$ be a preconditioner such that $\boldsymbol{\tilde{B}} = \boldsymbol{I}$. Then, the iterative scheme in \cref{eq:pmd} with a fixed stepsize $\alpha \equiv 1/2$ is equivalent to the power method for the GEP, ignoring normalization.
\end{theorem}
\begin{proof}
	Substituting a preconditioner $\boldsymbol{P}$ satisfying $\boldsymbol{\tilde{B}} = \boldsymbol{I}$ and a fixed stepsize $\alpha \equiv 1/2$ into the update rule in \cref{eq:pmd}, we obtain:
	$$
	\begin{aligned}
		\boldsymbol{x}_{k+1} = \boldsymbol{P}^{-1}\left(\boldsymbol{P}\boldsymbol{x}_k - \frac{1}{2} \left(2 \boldsymbol{P}^{-T}\boldsymbol{B}\boldsymbol{x}_{k} - \frac{\boldsymbol{P}^{-T}\boldsymbol{A}\boldsymbol{x}_k}{\left(\boldsymbol{x}_k^T\boldsymbol{A}\boldsymbol{x}_k\right)^{\frac{1}{2}}}\right)\right) = \frac{\boldsymbol{B}^{-1}\boldsymbol{Ax}_k}{2\left(\boldsymbol{x}_k^T\boldsymbol{A}\boldsymbol{x}_k\right)^{\frac{1}{2}}},
	\end{aligned}
	$$
	where the second equality uses the assumption $\boldsymbol{P}^T\boldsymbol{P} = \boldsymbol{B}$.
	
	This shows that the update is equivalent to:
	$$
	\boldsymbol{x}_{k+1} \propto \boldsymbol{B}^{-1} \boldsymbol{A} \boldsymbol{x}_k,
	$$
	which corresponds to the power method for the GEP, up to a normalization factor.
\end{proof}

From the perspective of PMD, the power method can essentially be viewed as a GD method in a transformed domain with a fixed stepsize $\alpha \equiv 1/2$. 
In fact, leveraging \cref{thm:global_convergence}, we can show that convergence is guaranteed for any stepsize $\alpha \in (0, 1)$:

\begin{corollary}
	Let $\boldsymbol{P}$ be a preconditioner such that $\boldsymbol{\tilde{B}} = \boldsymbol{I}$. 
	Consider the iterative scheme in \cref{eq:pmd} with a fixed stepsize $\alpha \equiv \alpha_0$, where $\alpha_0$ is sampled uniformly from the interval $(0,1)$. Then, the generated sequence $\left\{\boldsymbol{x}_k\right\}$ converges to a global minimizer of $f$ with probability one.
\end{corollary}

On one hand, we establish the convergence of the power method from a novel optimization perspective, distinguishing our analysis from classical approaches in numerical linear algebra \cite{pm_gep_1992}. 
On the other hand, this insight implies that the power method can potentially be accelerated by adopting a larger stepsize.
This also addresses the question raised at the end of \cref{sec:motivation_pmd}: by introducing the PMD framework in the transformed domain, we uncover the connection between the power method and first-order optimization methods, characterize the trade-off between the power method and GD method, and thereby develop algorithms that outperform the classical power method.

Although the optimal preconditioner can be obtained via Cholesky decomposition, in practice, we often seek a more efficient approximation by imposing structural constraints on $\boldsymbol{P}$. This leads to the following constrained optimization problem:
\begin{equation}
	\min_{\boldsymbol{P} \in \mathcal{C}} \| \boldsymbol{B} - \boldsymbol{P}^T\boldsymbol{P}\|_F^2,
	\label{eq:opt_P}
\end{equation}
where $\mathcal{C}$ denotes a set of structural constraints.

\begin{remark}
	(i) If no constraint is imposed, i.e., $\mathcal{C} = \mathbb{R}^{n \times n}$, the optimal solution to problem \cref{eq:opt_P} corresponds to the Cholesky factor of $\boldsymbol{B}$.
	
	(ii) If $\boldsymbol{P}$ is constrained to be a diagonal matrix, the optimal solution to problem \cref{eq:opt_P} can be obtained by simply extracting the diagonal entries of $\boldsymbol{B}$, i.e.,
	$$
	\boldsymbol{P} = \operatorname{diag}(\sqrt{b_{11}},\sqrt{b_{22}},\ldots,\sqrt{b_{nn}}),
	$$
	where $b_{ii}$ denotes the $(i,i)$-th diagonal entry of $\boldsymbol{B}$.
	
	(iii) If a sparsity pattern or a low-rank structure is imposed, an approximate $\boldsymbol{P}$ can be obtained via incomplete Cholesky factorization \cite{icf_2002}.
\end{remark}

In this section, we reveal the first-order nature of the power method. Inspired by the recent work \cite{liu2025difference}, we aim to further extend the Split-Merge approach to the GEP, thereby benefiting from richer second-order information. Additionally, we explore its intrinsic connection with the transform-domain framework.

\section{Split-Merge Algorithm for the GEP} \label{sec:sm}

The Split-Merge algorithm \cite{liu2025difference} has demonstrated superior efficiency in solving standard EP. A natural next step is to extend it to GEP, and we show that this extension is straightforward within the transform-domain framework.

\subsection{Splitting}

The core idea of the Split-Merge approach is to construct a tighter surrogate function at each iterate $\boldsymbol{x}_k$ by leveraging richer second-order information, which is achieved through a splitting operation. Specifically, we define the matrix
$$
\boldsymbol{H}_{\boldsymbol{x}} (\boldsymbol{u}, \boldsymbol{v}) = 2 \boldsymbol{B} - \frac{1}{\left(\boldsymbol{x}^T\boldsymbol{A}\boldsymbol{x}\right)^{\frac{1}{2}}} \boldsymbol{F}^T \left(\boldsymbol{u}\boldsymbol{u}^T + \boldsymbol{v} \boldsymbol{v}^T\right)\boldsymbol{F} + \frac{\left(\boldsymbol{Ax}\right) \left(\boldsymbol{Ax}\right)^T}{\left(\boldsymbol{x}^T\boldsymbol{A}\boldsymbol{x}\right)^{\frac{3}{2}}},
$$
where $\boldsymbol{F}$ is a full-rank factor of the PSD matrix $\boldsymbol{A}$, and the vectors $\boldsymbol{u}$ and $\boldsymbol{v}$ satisfy $\|\boldsymbol{u}\|\leq 1$, $\|\boldsymbol{v}\| \leq 1$, and $\boldsymbol{u}^T \boldsymbol{v} = 0$.

Following a similar line of analysis as in \cite{liu2025difference}, we obtain the following result:
\begin{theorem} \label{thm:H_succeq}
	For any PSD matrix $\boldsymbol{A}$ admitting a full-rank factorization $\boldsymbol{A} = \boldsymbol{F}^T \boldsymbol{F}$, and for any vectors $\boldsymbol{u}$ and $\boldsymbol{v}$ such that $\|\boldsymbol{u}\|\leq 1$, $\|\boldsymbol{v}\| \leq 1$, and $\boldsymbol{u}^T\boldsymbol{v}=0$, the following inequality holds for all $\boldsymbol{x}$:
	$$
	\boldsymbol{H}_{\boldsymbol{x}} (\boldsymbol{u}, \boldsymbol{v}) \succeq \nabla^2 f(\boldsymbol{x}).
	$$
\end{theorem}

Assuming that $\boldsymbol{H}_{\boldsymbol{x}_k} (\boldsymbol{u}, \boldsymbol{v}) \succ 0$, we can derive a family of iterative methods by varying $\boldsymbol{u}$ and $\boldsymbol{v}$:
\begin{equation}
	\boldsymbol{x}_{k+1} = \arg \min_{\boldsymbol{x}} \phi_k (\boldsymbol{x}),
	\label{eq:general_v}
\end{equation}
where $\phi_k (\boldsymbol{x})$ denotes a general quadratic surrogate function of $f(\boldsymbol{x})$ at the current iterate $\boldsymbol{x}_k$:
\begin{equation}
	\phi_k (\boldsymbol{x}) = f(\boldsymbol{x}_k) + \left\langle \nabla f(\boldsymbol{x}_k), \boldsymbol{x} - \boldsymbol{x}_k \right\rangle + \frac{1}{2} (\boldsymbol{x} - \boldsymbol{x}_k)^T \boldsymbol{H}_{\boldsymbol{x}_k}(\boldsymbol{u},\boldsymbol{v}) (\boldsymbol{x} - \boldsymbol{x}_k).
	\label{eq:phi_x}
\end{equation}

Following a similar strategy to that in \cite{liu2025difference}, we fix $\boldsymbol{u} \equiv \frac{\boldsymbol{F} \boldsymbol{x}_k}{\|\boldsymbol{F} \boldsymbol{x}_k\|}$. In this case, the matrix $\boldsymbol{H}_{\boldsymbol{x}} (\boldsymbol{u}, \boldsymbol{v})$ takes the form
$$
\boldsymbol{H}_{\boldsymbol{x}} (\boldsymbol{u}, \boldsymbol{v}) = 2 \boldsymbol{B} - \frac{1}{\left(\boldsymbol{x}^T\boldsymbol{A}\boldsymbol{x}\right)^{\frac{1}{2}}} \boldsymbol{F}^T\boldsymbol{v} \left(\boldsymbol{F}^T\boldsymbol{v}\right)^T
$$
Applying the Sherman-Morrison-Woodbury formula \cite{woodbury_1950}, we obtain the inverse:
\begin{equation}
	\left(\boldsymbol{H}_{\boldsymbol{x}_k} (\boldsymbol{u}, \boldsymbol{v})\right)^{-1} = \frac{1}{2} \left(\boldsymbol{B}^{-1} + \frac{1}{2\sigma (\boldsymbol{x}_k^T\boldsymbol{A} \boldsymbol{x}_k)^{\frac{1}{2}}} \boldsymbol{B}^{-1} \boldsymbol{F}^T \boldsymbol{v} (\boldsymbol{B}^{-1} \boldsymbol{F}^T \boldsymbol{v})^T\right), 
	\label{eq:H_inv}
\end{equation}
where $\sigma = 1 - \frac{\boldsymbol{v}^T \boldsymbol{F} \boldsymbol{B}^{-1} \boldsymbol{F}^T \boldsymbol{v} }{2 (\boldsymbol{x}_k^T \boldsymbol{A} \boldsymbol{x}_k)^{\frac{1}{2}}} > 0$. This condition ensures the positive definiteness of the matrix $\boldsymbol{H}_{\boldsymbol{x}_k} (\boldsymbol{u}, \boldsymbol{v})$.

Substituting the inverse from \cref{eq:H_inv} into the update formula in \cref{eq:general_v}, we obtain:
\begin{equation}
	\begin{aligned}
		\boldsymbol{x}_{k+1} &= \boldsymbol{x}_k - \left(\boldsymbol{H}_{\boldsymbol{x}_k} (\boldsymbol{u}, \boldsymbol{v})\right)^{-1} \nabla f(\boldsymbol{x}_k) \\
		&= \frac{1}{2(\boldsymbol{x}_k^T \boldsymbol{A} \boldsymbol{x}_k)^{\frac{1}{2}}} \boldsymbol{B}^{-1} \boldsymbol{Ax}_k + \frac{(\boldsymbol{B}^{-1} \boldsymbol{F}^T \boldsymbol{v})^T \boldsymbol{Ax}_k}{4\sigma (\boldsymbol{x}_k^T \boldsymbol{A} \boldsymbol{x}_k)} \boldsymbol{B}^{-1} \boldsymbol{F}^T \boldsymbol{v},
	\end{aligned}
	\label{eq:general_v_expand}
\end{equation}
where the final equality uses the orthogonality condition $\boldsymbol{u}^T \boldsymbol{v} = 0$.

\begin{remark} \label{rmk:reduce_pm}
	When $\boldsymbol{v} = \boldsymbol{0}$, the update rule in \cref{eq:general_v_expand} simplifies to the classical power method for the GEP. In other words, the power method can be viewed as a special case within this family of algorithms.
\end{remark}

For a general choice of $\boldsymbol{v}$, the resulting surrogate function is tighter than those used in both GD and the classical power method. This relationship is formalized in the following proposition:
\begin{proposition} \label{prop:tigher_pm}
	Let $\boldsymbol{u} \equiv \frac{\boldsymbol{F} \boldsymbol{x}_k}{\|\boldsymbol{F} \boldsymbol{x}_k\|}$. For any vector $\boldsymbol{v}$ satisfying $\|\boldsymbol{v}\| \leq 1$ and $\boldsymbol{u}^T\boldsymbol{v}=0$, it holds that
	$$
	2 \lambda_1(\boldsymbol{B}) \boldsymbol{I} \succeq 2 \boldsymbol{B} \succeq \boldsymbol{H}_{\boldsymbol{x}_k} (\boldsymbol{u}, \boldsymbol{v}) \succeq \nabla^2 f(\boldsymbol{x}_k).
	$$
\end{proposition}

However, for a general choice of $\boldsymbol{v}$, this strategy relies on the decomposition $\boldsymbol{A} = \boldsymbol{F}^T \boldsymbol{F}$, which is often undesirable. One of the most elegant aspects of the Split-Merge method \cite{liu2025difference} is its ability to eliminate this dependence by selecting an optimal $\boldsymbol{v}$ that effectively merges the decomposition, resulting in a decomposition-free algorithm. This idea can be naturally extended to the GEP setting.

\subsection{Merging}

We adopt the same strategy as in \cite{liu2025difference} to select $\boldsymbol{v}$ so that the surrogate function $\phi_k(\boldsymbol{x})$ in \cref{eq:phi_x} achieves the maximum reduction:
\begin{equation}
	\hat{\boldsymbol{v}} = \arg \min_{\boldsymbol{v}} \left\{\min_{\boldsymbol{d} \in \mathbb{R}^n} \phi_k(\boldsymbol{x}_k + \boldsymbol{d})\right\},\ \text{s.t.} \  \boldsymbol{v} \in \Omega,
	\label{eq:optimal_v}
\end{equation}
where $\boldsymbol{d}$ denotes the search direction used to update the iterate, i.e., $\boldsymbol{x}_{k+1} = \boldsymbol{x}_{k}+ \boldsymbol{d}$. The set $\Omega$ is defined as
$$
\Omega = \left\{ \boldsymbol{v}: \boldsymbol{v}^T\boldsymbol{u} = 0, \boldsymbol{v}^T \boldsymbol{v} = \frac{1}{\rho} \right\},
$$
where $\rho \geq 1$ is a normalization parameter chosen to ensure that $\boldsymbol{H}_{\boldsymbol{x}_k} (\boldsymbol{u}, \boldsymbol{v}) \succ 0$ (i.e., $\sigma > 0$).

Nevertheless, problem \cref{eq:optimal_v} remains difficult to solve, as it reduces to a new GEP, as stated in the following theorem:
\begin{theorem} \label{thm:gep}
	The optimal solution of problem \cref{eq:optimal_v} is equivalent to solving the following GEP:
	\begin{equation}
		\hat{\boldsymbol{v}} = \arg \max_{\boldsymbol{v}} \frac{\boldsymbol{v}^T \boldsymbol{D} \boldsymbol{v}}{\boldsymbol{v}^T \boldsymbol{C} \boldsymbol{v}},\ \text{s.t.} \  \boldsymbol{v} \in \Omega,
		\label{eq:gep}
	\end{equation}
	where $\boldsymbol{D} =  \boldsymbol{q} \boldsymbol{q}^T \succeq 0,\ \boldsymbol{q} = \boldsymbol{F}\boldsymbol{B}^{-1}\boldsymbol{A}\boldsymbol{x}_k,$ and
	$$
	\boldsymbol{C} = \rho \boldsymbol{I} - \frac{\boldsymbol{F}\boldsymbol{B}^{-1}\boldsymbol{F}^T}{2(\boldsymbol{x}_k^T \boldsymbol{A} \boldsymbol{x}_k)^{\frac{1}{2}}} \succ 0.
	$$
\end{theorem}

\begin{proof}
	Given that the matrix $\boldsymbol{H}_{\boldsymbol{x}_k} (\boldsymbol{u},\boldsymbol{v})$ is PD, we first consider the update direction for a fixed $\boldsymbol{v}$. Specifically, we solve the subproblem
	$$
	\begin{aligned}
		\hat{\boldsymbol{d}} &= \arg \min_{\boldsymbol{d} \in \mathbb{R}^n} \phi_k(\boldsymbol{x}_k + \boldsymbol{d}) \\
		&= \arg \min_{\boldsymbol{d} \in \mathbb{R}^n} \nabla f(\boldsymbol{x}_k)^T\boldsymbol{d} + \frac{1}{2}\boldsymbol{d}^T \boldsymbol{H}_{\boldsymbol{x}_k} (\boldsymbol{u},\boldsymbol{v}) \boldsymbol{d} \\
		&= - \left(\boldsymbol{H}_{\boldsymbol{x}_k} (\boldsymbol{u},\boldsymbol{v})\right)^{-1} \nabla f(\boldsymbol{x}_k).
	\end{aligned}
	$$
	Substituting this expression into the outer optimization problem \cref{eq:optimal_v}, we derive
	$$
	\hat{\boldsymbol{v}} = \arg \max_{\boldsymbol{v}} \frac{1}{2} \nabla f(\boldsymbol{x}_k)^T \left(\boldsymbol{H}_{\boldsymbol{x}_k} (\boldsymbol{u},\boldsymbol{v})\right)^{-1}  \nabla f(\boldsymbol{x}_k),\ \text{s.t.} \  \boldsymbol{v} \in \Omega.
	$$
	To proceed, we substitute the explicit forms of $\nabla f(\boldsymbol{x})$ and $\boldsymbol{H}_{\boldsymbol{x}_k} (\boldsymbol{u},\boldsymbol{v})^{-1}$ from \cref{eq:grad_f} and \cref{eq:H_inv}, respectively. Under the orthogonality constraint $\boldsymbol{v}^T\boldsymbol{Fx}_k=0$, the objective simplifies to:
	$$
	\hat{\boldsymbol{v}} = \arg \max_{\boldsymbol{v}} \frac{\left(\boldsymbol{v}^T \boldsymbol{F}\boldsymbol{B}^{-1}\boldsymbol{A}\boldsymbol{x}_k\right)^2}{1 - \frac{\boldsymbol{v}^T \boldsymbol{F}\boldsymbol{B}^{-1}\boldsymbol{F}^T \boldsymbol{v}}{2(\boldsymbol{x}_k^T \boldsymbol{A} \boldsymbol{x}_k)^{\frac{1}{2}}}},\ \text{s.t.} \  \boldsymbol{v} \in \Omega.
	$$
	Applying the regularization constraint $\boldsymbol{v}^T\boldsymbol{v} = \frac{1}{\rho}$, the expression can be further rewritten as:
	$$
	\hat{\boldsymbol{v}} = \arg \max_{\boldsymbol{v}} \frac{\left(\boldsymbol{v}^T \boldsymbol{F}\boldsymbol{B}^{-1}\boldsymbol{A}\boldsymbol{x}_k\right)^2}{\rho \boldsymbol{v}^T\boldsymbol{v} - \frac{\boldsymbol{v}^T \boldsymbol{FF}^T \boldsymbol{v}}{2(\boldsymbol{x}_k^T \boldsymbol{A} \boldsymbol{x}_k)^{\frac{1}{2}}}},\ \text{s.t.} \  \boldsymbol{v} \in \Omega.
	$$
	Finally, this expression can be compactly formulated as a generalized Rayleigh quotient:
	$$
	\hat{\boldsymbol{v}} = \arg \max_{\boldsymbol{v}} \frac{\boldsymbol{v}^T \boldsymbol{D} \boldsymbol{v}}{\boldsymbol{v}^T \boldsymbol{C} \boldsymbol{v}},\ \text{s.t.} \  \boldsymbol{v} \in \Omega.
	$$
	Hence, the proof is complete.
\end{proof}

To mitigate the challenge of selecting the optimal $\hat{\boldsymbol{v}}$ in \cref{eq:gep}, we relax the objective by considering the following bounds:
$$
\left(\boldsymbol{v}^T \boldsymbol{F}\boldsymbol{B}^{-1}\boldsymbol{A}\boldsymbol{x}_k\right)^2 \leq \frac{\boldsymbol{v}^T \boldsymbol{D} \boldsymbol{v}}{\boldsymbol{v}^T \boldsymbol{C} \boldsymbol{v}} \leq \frac{\left(\boldsymbol{v}^T \boldsymbol{F}\boldsymbol{B}^{-1}\boldsymbol{A}\boldsymbol{x}_k\right)^2}{1 - \frac{\lambda_1}{2\rho(\boldsymbol{x}_k^T \boldsymbol{A} \boldsymbol{x}_k)^{\frac{1}{2}}}},
$$
where $\lambda_1$ is the dominant generalized eigenvalue of the matrix pair $\left(\boldsymbol{A},\boldsymbol{B}\right)$.

This relaxation yields the following optimization problem for selecting $\boldsymbol{v}$:
$$
\max_{\boldsymbol{v}} \boldsymbol{v}^T \boldsymbol{F}\boldsymbol{B}^{-1}\boldsymbol{A}\boldsymbol{x}_k,\ \text{s.t.} \  \boldsymbol{v} \in \Omega.
$$ 
According to the Karush-Kuhn-Tucker conditions \cite{boyd_cvx}, this problem admits a closed-form solution:
$$
\hat{\boldsymbol{v}} = \frac{\boldsymbol{F}\boldsymbol{B}^{-1}\boldsymbol{A}\boldsymbol{x}_k - \frac{\boldsymbol{x}_k^T \boldsymbol{A}\boldsymbol{B}^{-1}\boldsymbol{A} \boldsymbol{x}_k}{\boldsymbol{x}_k^T \boldsymbol{A} \boldsymbol{x}_k} \boldsymbol{Fx}_k}{\sqrt{\rho} \left\|\boldsymbol{F}\boldsymbol{B}^{-1}\boldsymbol{A}\boldsymbol{x}_k - \frac{\boldsymbol{x}_k^T \boldsymbol{A}\boldsymbol{B}^{-1}\boldsymbol{A} \boldsymbol{x}_k}{\boldsymbol{x}_k^T \boldsymbol{A} \boldsymbol{x}_k} \boldsymbol{Fx}_k\right\|}.
$$
Substituting this optimal $\hat{\boldsymbol{v}}$ into the update rule in \cref{eq:general_v_expand} yields the following iteration:
\begin{equation}
	\boldsymbol{x}_{k+1} = \zeta_k \boldsymbol{B}^{-1}\boldsymbol{A}\boldsymbol{x}_k + \omega_k \boldsymbol{B}^{-1} \boldsymbol{A} \boldsymbol{B}^{-1} \boldsymbol{A} \boldsymbol{x}_k,
	\label{eq:kkt_merge}
\end{equation}
where the scalar coefficients are given by
$$
\begin{aligned}
	&\zeta_k = \frac{1}{2(\boldsymbol{x}_k^T \boldsymbol{A} \boldsymbol{x}_k)^{\frac{1}{2}}} - \frac{1}{4\sigma_k\rho(\boldsymbol{x}_k^T \boldsymbol{A} \boldsymbol{x}_k)} \frac{\boldsymbol{x}_k^T\boldsymbol{A}\boldsymbol{B}^{-1}\boldsymbol{A}\boldsymbol{x}_k}{\boldsymbol{x}_k^T\boldsymbol{A}\boldsymbol{x}_k}, \\
	&\omega_k = \frac{1}{4\sigma_k\rho(\boldsymbol{x}_k^T \boldsymbol{A} \boldsymbol{x}_k)}.
\end{aligned}
$$
Moreover, the parameter $\sigma_k$ is defined as
$$
\sigma_k = 1 - \frac{\boldsymbol{z}_k^T\boldsymbol{B}^{-1}\boldsymbol{z}_k}{2\rho(\boldsymbol{x}_k^T \boldsymbol{A} \boldsymbol{x}_k)^{\frac{1}{2}} \delta_k},
$$
where
$$
\boldsymbol{z}_k = \boldsymbol{A}\boldsymbol{B}^{-1}\boldsymbol{A} \boldsymbol{x}_k - \frac{\boldsymbol{x}_k^T \boldsymbol{A}\boldsymbol{B}^{-1}\boldsymbol{A} \boldsymbol{x}_k}{\boldsymbol{x}_k^T \boldsymbol{A} \boldsymbol{x}_k} \boldsymbol{A} \boldsymbol{x}_k,
$$
and
$$
\delta_k = \boldsymbol{x}_k^T \boldsymbol{A}\boldsymbol{B}^{-1}\boldsymbol{A}\boldsymbol{B}^{-1}\boldsymbol{A} \boldsymbol{x}_k - \frac{(\boldsymbol{x}_k^T \boldsymbol{A}\boldsymbol{B}^{-1}\boldsymbol{A} \boldsymbol{x}_k)^2}{\boldsymbol{x}_k^T \boldsymbol{A} \boldsymbol{x}_k}.
$$
Notably, this iterative scheme completely avoids the explicit decomposition of $\boldsymbol{A} = \boldsymbol{F}^T \boldsymbol{F}$, rendering the method decomposition-free and computationally more efficient.

\textbf{Computational Cost}:
By leveraging the symmetry of the matrix pair $(\boldsymbol{A}, \boldsymbol{B})$, the computational cost can be significantly reduced by avoiding redundant calculations. Each Split-Merge step requires solving 2 linear systems, performing 2 matrix-vector products, and computing 4 vector-vector products. The linear system solutions can be further accelerated by precomputing and storing the Cholesky decomposition or by employing approximate solvers such as the preconditioned conjugate gradient (PCG) method \cite{xu2020practical}.

More importantly, we will further show within the transform-domain framework that this extension of the Split-Merge method from the standard EP to the GEP is both natural and straightforward, as demonstrated in the following theorem.
\begin{theorem} \label{thm:sm_ep_gep}
	The Split-Merge algorithm for the GEP is equivalent to applying the Split-Merge algorithm to the standard EP in a transformed domain, under a preconditioner $\boldsymbol{P}$ such that $\boldsymbol{\tilde{B}} = \boldsymbol{I}$.
\end{theorem}

\begin{proof}
	Suppose there exists a preconditioner $\boldsymbol{P}$ satisfying $\boldsymbol{\tilde{B}} = \boldsymbol{I}$.
	Then, the original GEP reduces to a standard EP in the transformed domain $\boldsymbol{y} = \boldsymbol{P} \boldsymbol{x}$, namely:
	$$
	\min_{\boldsymbol{y} \in \mathbb{R}^n} \boldsymbol{y}^T\boldsymbol{y} - \left(\boldsymbol{y}^T \boldsymbol{\tilde{A}} \boldsymbol{y}\right)^{\frac{1}{2}}.
	$$
	Applying the Split-Merge algorithm to the standard EP in this transformed domain \cite{liu2025difference} yields:
	$$
	\boldsymbol{y}_{k+1} =  \tilde{\zeta}_k \boldsymbol{\tilde{A}}\boldsymbol{y}_k + \tilde{\omega}_k \boldsymbol{\tilde{A}}^2\boldsymbol{y}_k
	$$
	where
	$$
	\begin{aligned}
		&\tilde{\zeta}_k = \frac{1}{2(\boldsymbol{y}_k^T \boldsymbol{\tilde{A}} \boldsymbol{y}_k)^{\frac{1}{2}}} - \frac{1}{4\tilde{\sigma}_k\rho(\boldsymbol{y}_k^T \boldsymbol{\tilde{A}} \boldsymbol{y}_k)} \frac{\boldsymbol{y}_k^T\boldsymbol{\tilde{A}}^2\boldsymbol{y}_k}{\boldsymbol{y}_k^T\boldsymbol{\tilde{A}}\boldsymbol{y}_k}, \\
		&\tilde{\omega}_k = \frac{1}{4\tilde{\sigma}_k\rho(\boldsymbol{y}_k^T \boldsymbol{\tilde{A}} \boldsymbol{y}_k)}.
	\end{aligned}
	$$
	Here, the coefficient $\tilde{\sigma}_k$ is computed as:,
	$$
	\tilde{\sigma}_k = 1 - \frac{\left\| \boldsymbol{\tilde{A}}^2 \boldsymbol{y}_k - \frac{\boldsymbol{y}_k^T \boldsymbol{\tilde{A}}^2 \boldsymbol{y}_k}{\boldsymbol{y}_k^T \boldsymbol{\tilde{A}} \boldsymbol{y}_k} \boldsymbol{\tilde{A}} \boldsymbol{y}_k \right\|^2}{2\rho(\boldsymbol{y}_k^T \boldsymbol{\tilde{A}} \boldsymbol{y}_k)^{\frac{1}{2}} \left(\boldsymbol{y}_k^T \boldsymbol{\tilde{A}}^3 \boldsymbol{y}_k - \frac{(\boldsymbol{y}_k^T \boldsymbol{\tilde{A}}^2 \boldsymbol{y}_k)^2}{\boldsymbol{y}_k^T \boldsymbol{\tilde{A}} \boldsymbol{y}_k}\right)}.
	$$
	Substituting $\boldsymbol{y}_{k} = \boldsymbol{P} \boldsymbol{x}_{k}$ and $\boldsymbol{x}_{k+1} = \boldsymbol{P}^{-1} \boldsymbol{y}_{k+1}$, this procedure is equivalent to applying the Split-Merge algorithm to the GEP as formulated in \cref{eq:kkt_merge}.
\end{proof}

\begin{remark}
	Utilizing the result of \cref{thm:sm_ep_gep}, we can directly extend the convergence analysis of the Split-Merge algorithm for EP \cite{liu2025difference} to the GEP setting via the transform-domain framework.
\end{remark}

\begin{table} [t]
	\centering
	\caption{\textbf{Comparison of average running time (in seconds) across different matrix dimensions ($n$) and condition numbers of matrix $\boldsymbol{B}$ ($\kappa_{\boldsymbol{B}}$)}. The best results are highlighted in \textbf{bold}, and the second-best are \underline{underlined}. We also report the runtime of MATLAB's built-in solver \texttt{eigs} in the second-to-last column. If the Split-Merge algorithm runs faster than \texttt{eigs}, then the \texttt{eigs} runtime is marked with \uwave{a wavy underline}. The last column presents the speedup of Split-Merge over the power method, highlighted in \red{\textbf{red bold}}. If an algorithm fails to converge in some random trials, its runtime is annotated with \red{\cha}; averages are then computed only over the successful runs to ensure a fair pairwise comparison.}
	\label{tab:sota_condition_number}%
	\resizebox{1\linewidth}{!}
	{
		\begin{tabular}{cc|ccccc|c|c}
			\toprule
			$n$ & \multicolumn{1}{c}{$\kappa_{\boldsymbol{B}}$} & \textbf{GD} & \textbf{Power} & \textbf{PMD (Cholesky)} & \textbf{Lanczos} & \multicolumn{1}{c}{\textbf{Split-Merge}} & \multicolumn{1}{c}{\textbf{\texttt{Eigs}}} & \textbf{Speedup} \\
			\midrule
			\multirow{10}[1]{*}{\textbf{256}} & \textbf{3 } & 1.42E-02 & 1.97E-02 & \underline{1.21E-02} & 1.90E-02 & \textbf{5.23E-03} & \textit{\uwave{1.69E-02}} & \textcolor[rgb]{ 1,  0,  0}{\textbf{376.22\%}} \\
			& \textbf{5 } & 2.04E-02 & 2.52E-02 & 1.56E-02 & \underline{1.50E-02} & \textbf{6.78E-03} & \textit{\uwave{9.77E-03}} & \textcolor[rgb]{ 1,  0,  0}{\textbf{372.21\%}} \\
			& \textbf{8 } & 1.33E-02 & 6.61E-03 & \underline{4.15E-03} & 9.78E-03 & \textbf{2.35E-03} & \textit{\uwave{6.92E-03}} & \textcolor[rgb]{ 1,  0,  0}{\textbf{281.84\%}} \\
			& \textbf{10 } & 1.66E-02 & 9.12E-03 & \underline{5.74E-03} & 9.36E-03 & \textbf{2.93E-03} & \textit{\uwave{2.38E-02}} & \textcolor[rgb]{ 1,  0,  0}{\textbf{311.11\%}} \\
			& \textbf{13 } & 1.65E-02 & 6.98E-03 & \underline{4.36E-03} & 7.15E-03 & \textbf{2.48E-03} & \textit{\uwave{4.55E-03}} & \textcolor[rgb]{ 1,  0,  0}{\textbf{281.69\%}} \\
			& \textbf{30 } & 1.75E-02 & 3.38E-03 & \underline{2.18E-03} & 3.82E-03 & \textbf{1.55E-03} & \textit{\uwave{3.54E-03}} & \textcolor[rgb]{ 1,  0,  0}{\textbf{217.89\%}} \\
			& \textbf{40 } & 2.54E-02 & 4.36E-03 & \underline{2.83E-03} & 3.75E-03 & \textbf{1.85E-03} & \textit{\uwave{4.99E-03}} & \textcolor[rgb]{ 1,  0,  0}{\textbf{235.20\%}} \\
			& \textbf{50 } & 2.89E-02 & 4.44E-03 & \underline{2.85E-03} & 3.21E-03 & \textbf{1.76E-03} & \textit{\uwave{4.11E-03}} & \textcolor[rgb]{ 1,  0,  0}{\textbf{252.40\%}} \\
			& \textbf{80 } & 2.73E-02 & 2.72E-03 & \underline{1.77E-03} & 1.94E-03 & \textbf{1.37E-03} & \textit{\uwave{4.44E-03}} & \textcolor[rgb]{ 1,  0,  0}{\textbf{198.48\%}} \\
			& \textbf{100 } & 1.87E-02 & 1.42E-03 & \textbf{9.68E-04} & 1.94E-03 & \underline{9.79E-04} & \textit{\uwave{4.18E-03}} & \textcolor[rgb]{ 1,  0,  0}{\textbf{144.68\%}} \\
			\midrule
			\multirow{10}[0]{*}{\textbf{512}} & \textbf{3 } & 4.28E-02 & 1.03E-01 & \underline{6.42E-02} & 1.03E+00 & \textbf{2.69E-02} & 1.56E-02 & \textcolor[rgb]{ 1,  0,  0}{\textbf{384.30\%}} \\
			& \textbf{5 } & 1.75E-01 & 3.38E-01 & 2.06E-01 & \textbf{3.84E-02} & \underline{8.80E-02} & 1.57E-02 & \textcolor[rgb]{ 1,  0,  0}{\textbf{383.73\%}} \\
			& \textbf{8 } & 6.30E-02 & 7.03E-02 & 4.37E-02 & \underline{2.50E-02} & \textbf{2.13E-02} & 1.33E-02 & \textcolor[rgb]{ 1,  0,  0}{\textbf{330.62\%}} \\
			& \textbf{10 } & 1.18E-01 & 1.13E-01 & 7.12E-02 & \textbf{2.79E-02} & \underline{3.08E-02} & 1.45E-02 & \textcolor[rgb]{ 1,  0,  0}{\textbf{368.86\%}} \\
			& \textbf{13 } & 4.74E-02 & 2.91E-02 & \underline{1.84E-02} & 2.17E-02 & \textbf{9.96E-03} & \textit{\uwave{1.24E-02}} & \textcolor[rgb]{ 1,  0,  0}{\textbf{292.22\%}} \\
			& \textbf{30 } & 6.21E-02 & 1.68E-02 & \underline{1.09E-02} & 1.21E-02 & \textbf{6.96E-03} & \textit{\uwave{1.03E-02}} & \textcolor[rgb]{ 1,  0,  0}{\textbf{242.00\%}} \\
			& \textbf{40 } & 6.57E-02 & 1.56E-02 & 1.00E-02 & \underline{9.69E-03} & \textbf{6.89E-03} & \textit{\uwave{9.53E-03}} & \textcolor[rgb]{ 1,  0,  0}{\textbf{226.49\%}} \\
			& \textbf{50 } & 1.47E-01 & 2.57E-02 & 1.65E-02 & \underline{9.28E-03} & \textbf{9.23E-03} & \textit{\uwave{1.09E-02}} & \textcolor[rgb]{ 1,  0,  0}{\textbf{278.04\%}} \\
			& \textbf{80 } & 1.33E-01 & 1.56E-02 & 1.01E-02 & \underline{7.44E-03} & \textbf{6.51E-03} & \textit{\uwave{8.39E-03}} & \textcolor[rgb]{ 1,  0,  0}{\textbf{239.22\%}} \\
			& \textbf{100 } & 8.98E-02 & 8.99E-03 & 5.98E-03 & \underline{5.60E-03} & \textbf{4.86E-03} & \textit{\uwave{9.58E-03}} & \textcolor[rgb]{ 1,  0,  0}{\textbf{185.05\%}} \\
			\midrule
			\multirow{10}[1]{*}{\textbf{1024}} & \textbf{3 } & 3.46E-01 & 2.01E+00 & 1.29E+00 & \textbf{1.55E-01} (\red{\cha}) & \underline{4.97E-01} & 5.89E-02 & \textcolor[rgb]{ 1,  0,  0}{\textbf{404.90\%}} \\
			& \textbf{5 } & 1.43E-01 & 4.10E-01 & 2.58E-01 & \underline{1.11E-01} (\red{\cha}) & \textbf{1.09E-01} & 4.07E-02 & \textcolor[rgb]{ 1,  0,  0}{\textbf{375.31\%}} \\
			& \textbf{8 } & 1.45E-01 & 2.70E-01 & 1.72E-01 & \underline{8.40E-02} & \textbf{7.60E-02} & 4.38E-02 & \textcolor[rgb]{ 1,  0,  0}{\textbf{355.97\%}} \\
			& \textbf{10 } & 1.38E-01 & 2.00E-01 & 1.28E-01 & \underline{8.41E-02} & \textbf{5.97E-02} & 4.00E-02 & \textcolor[rgb]{ 1,  0,  0}{\textbf{335.42\%}} \\
			& \textbf{13 } & 2.83E-01 & 3.74E-01 & 2.37E-01 & \textbf{7.55E-02} & \underline{1.04E-01} & 3.56E-02 & \textcolor[rgb]{ 1,  0,  0}{\textbf{359.38\%}} \\
			& \textbf{30 } & 2.90E-01 & 1.62E-01 & 1.03E-01 & \textbf{4.14E-02} & \underline{5.07E-02} & 3.12E-02 & \textcolor[rgb]{ 1,  0,  0}{\textbf{319.38\%}} \\
			& \textbf{40 } & 3.11E-01 & 1.29E-01 & 8.20E-02 & \textbf{3.55E-02} & \underline{4.06E-02} & 2.54E-02 & \textcolor[rgb]{ 1,  0,  0}{\textbf{316.60\%}} \\
			& \textbf{50 } & 5.73E-01 & 2.00E-01 & 1.28E-01 & \textbf{3.57E-02} & \underline{5.95E-02} & 3.48E-02 & \textcolor[rgb]{ 1,  0,  0}{\textbf{335.76\%}} \\
			& \textbf{80 } & 2.68E-01 & 5.50E-02 & 3.62E-02 & \textbf{2.13E-02} & \underline{2.36E-02} & \textit{\uwave{3.04E-02}} & \textcolor[rgb]{ 1,  0,  0}{\textbf{232.99\%}} \\
			& \textbf{100 } & 3.84E-01 & 6.58E-02 & 4.24E-02 & \textbf{2.03E-02} & \underline{2.58E-02} & \textit{\uwave{2.82E-02}} & \textcolor[rgb]{ 1,  0,  0}{\textbf{254.74\%}} \\
			\bottomrule
		\end{tabular}%
	}
\end{table}

\section{Numerical Experiments} \label{sec:simulation}

In this section, we present experimental results for solving the $1$-GEP on both synthetic and real-world datasets, aiming to validate the proposed methods in terms of computational efficiency and numerical stability. All experiments were conducted using MATLAB R2021b on a Windows system, running on an Alienware x17 R2 laptop with an Intel i7-12700H processor (2.30 GHz) and 16 GB of RAM.

\subsection{Synthetic Dataset}

\begin{figure}[t]
	\centering
	\begin{subfigure}[b]{0.24\linewidth}
		\centering
		\includegraphics[width=\textwidth]{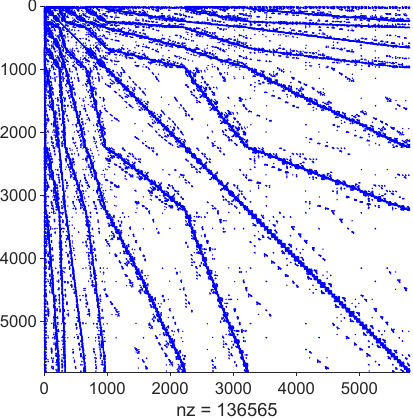}
		\caption{Lapla3\_A}
		\label{fig:lapla3_A}
	\end{subfigure}
	\hfill
	\begin{subfigure}[b]{0.24\linewidth}
		\centering
		\includegraphics[width=\textwidth]{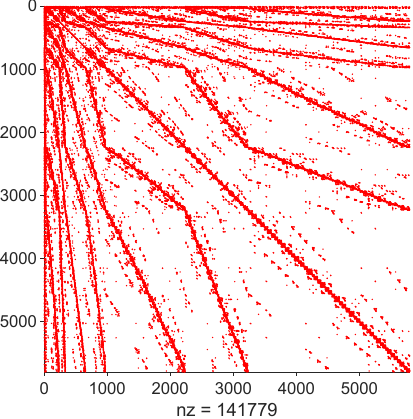}
		\caption{Lapla3\_B}
		\label{fig:lapla3_B}
	\end{subfigure}
	\hfill
	\begin{subfigure}[b]{0.24\linewidth}
		\centering
		\includegraphics[width=\textwidth]{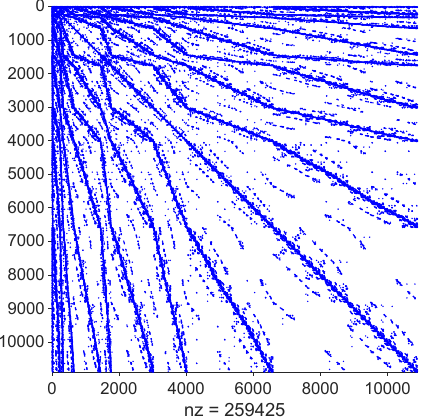}
		\caption{Lapla4\_A}
		\label{fig:lapla4_A}
	\end{subfigure}
	\hfill
	\begin{subfigure}[b]{0.24\linewidth}
		\centering
		\includegraphics[width=\textwidth]{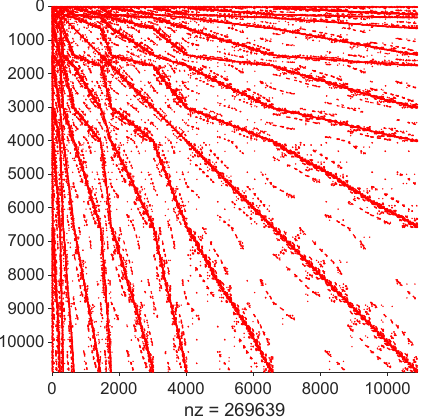}
		\caption{Lapla4\_B}
		\label{fig:lapla4_B}
	\end{subfigure}
	\caption{\textbf{Sparsity patterns of the two tested matrix pairs.} Dimensions: (a, b) Lapla3 with size 5,795; (c, d) Lapla4 with size 10,891. Number of nonzeros (nnz): (a) Lapla3\_A: 136,565; (b) Lapla3\_B: 141,779; (c) Lapla4\_A: 259,425; (d) Lapla4\_B: 269,639.}
	\label{fig:clustered_dataset}
\end{figure}

We begin by evaluating the performance of the proposed methods on synthetic datasets. The matrix pair $(\boldsymbol{A}, \boldsymbol{B})$ is randomly generated with a fixed condition number $\kappa_{\boldsymbol{A}}$, while varying the condition number $\kappa_{\boldsymbol{B}}$ and matrix dimension $n$ (see \cref{appendix:results_condition_number} for details).

We compare our methods with several established benchmarks, including the GD method \cite{auchmuty1991globally}, the power method \cite{pm_gep_1992}, and the Lanczos method \cite{gep_2011}. MATLAB's built-in solver \texttt{eigs} is also included as a baseline due to its efficiency and reliability in scientific computing.

The GD method is implemented with a fixed stepsize sampled uniformly at random from the interval $[0.9/\lambda_1(\boldsymbol{B}), 0.99/\lambda_1(\boldsymbol{B})]$, where $\lambda_1(\boldsymbol{B})$ is computed using \texttt{eigs}. We also evaluate the PMD method with a preconditioner $\boldsymbol{P} = \boldsymbol{L}^T$, where $\boldsymbol{L}$ is the Cholesky factor of $\boldsymbol{B}$. To ensure a fair comparison, the PMD method uses a fixed stepsize sampled uniformly at random from the interval $[0.9, 0.99]$, consistent with that of the GD method. For all methods, all linear systems involving $\boldsymbol{B}$ are solved using a pre-stored Cholesky decomposition.

The stopping criterion is defined as $\sin(\theta_k) \leq \epsilon$, where $\epsilon = 10^{-5}$ and $\theta_k$ is the angle between the $k$-th iterate $\boldsymbol{x}_k$ and the ground-truth eigenvector $\boldsymbol{u}_1$, obtained via \texttt{eigs}. Alternatively, the algorithm terminates when the number of iterations reaches the maximum limit of 100,000.
All experiments are repeated 100 times with different random initializations of $\boldsymbol{x}_0$, generated using MATLAB's \texttt{randn} function and kept identical across all methods. The average running time is used as the primary performance metric.

\cref{tab:sota_condition_number} reports the average running time across different matrix dimensions ($n$) and condition numbers of matrix $\boldsymbol{B}$ ($\kappa_{\boldsymbol{B}}$). The proposed Split-Merge method consistently outperforms existing benchmarks in nearly all settings. This advantage is particularly pronounced when $n = 256$, even compared to the subspace-based Lanczos method, which is known for its optimal per-iteration complexity.

Although Lanczos achieves slightly better average performance than Split-Merge when $n = 1024$, it exhibits numerical instability in certain cases. Specifically, when $\kappa_{\boldsymbol{B}} = 3$ and $5$, the success rates of convergence are only 55\% and 99\%, respectively. These cases involve tightly clustered eigenvalues and small eigen-gaps, which pose challenges for convergence. Such instability undermines the reliability of Lanczos in practical applications. In contrast, Split-Merge achieves successful convergence in all tested cases.

In addition, Split-Merge demonstrates efficiency comparable to that of MATLAB's built-in solver \texttt{eigs}. This is especially evident when $n$ is small, where Split-Merge consistently achieves lower average runtime. The results also show that Split-Merge consistently outperforms the power method, achieving a maximum observed speedup of 404.90\%.

The table further reveals that when $\kappa_{\boldsymbol{B}} < 10$, the GD method outperforms the power method. This supports the observation that in well-conditioned cases, the surrogate function used by GD, which depends only on the dominant eigenvalue of $\boldsymbol{B}$, already provides sufficient accuracy. Since GD does not require solving linear systems, it exhibits clear advantages in computational efficiency under such conditions.

Finally, the PMD method using the Cholesky-based preconditioner $\boldsymbol{P} = \boldsymbol{L}^T$ consistently outperforms the power method. This confirms that employing larger stepsizes in the transformed domain can significantly accelerate convergence. In particular, when $n = 256$, the PMD (Cholesky) method achieves performance that is close to optimal among all tested methods.

\begin{figure}[t]
	\centering
	\begin{subfigure}[b]{0.45\linewidth}
		\centering
		\includegraphics[width=\textwidth]{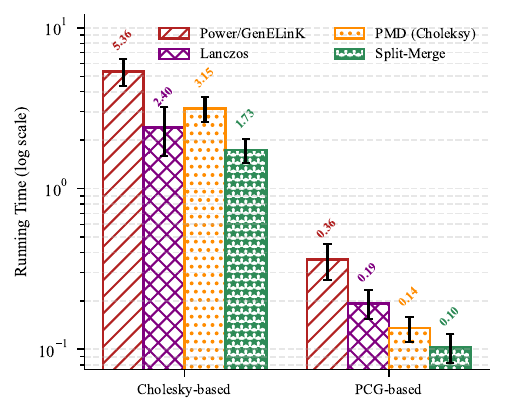}
		\caption{Lapla3}
		\label{fig:barplo_Lapla3}
	\end{subfigure}
	\hfill
	\begin{subfigure}[b]{0.45\linewidth}
		\centering
		\includegraphics[width=\textwidth]{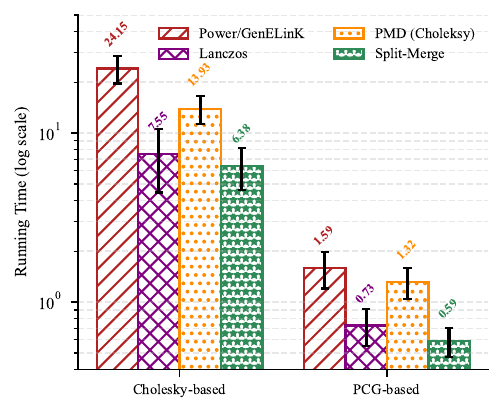}
		\caption{Lapla4}
		\label{fig:barplo_Lapla4}
	\end{subfigure}
	\caption{\textbf{Performance comparison of different methods on two real-world clustered matrix datasets.} (a) and (b) show the average running time (in seconds, on a logarithmic scale) for the Lapla3 and Lapla4 matrix pairs, respectively.}
	\label{fig:barplot_clustered}
\end{figure}

\begin{figure}[t]
	\centering
	\includegraphics[width=\linewidth]{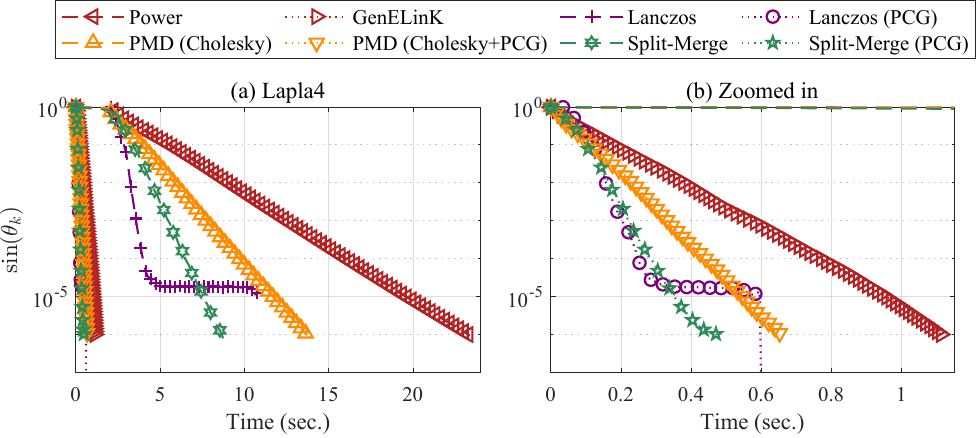}
	\caption{\textbf{Convergence comparison of different methods on the Lapla4 matrix pair}. (a) Overall comparison between Cholesky-based and PCG-based methods. (b) Zoomed-in view highlighting the performance of PCG-based methods.}
	\label{fig:sin_time_Lapla}
\end{figure}

\subsection{Real-World Dataset}

We further evaluate the performance of different algorithms on more challenging real-world datasets\footnote{\url{https://s2.smu.edu/yzhou/data/matrices.htm}}. These datasets exhibit clustered eigenvalue distributions, resulting in smaller eigen-gaps and thus making recovery more difficult. The sparsity patterns and statistical properties of the two selected large-scale matrices are illustrated in \cref{fig:clustered_dataset}.

For large-scale matrices with sparse structures, a key computational bottleneck lies in solving linear systems. Recent studies (e.g., GenELinK \cite{ge2016efficient}) have suggested that employing inexact solvers can significantly reduce the computational cost of this subproblem.

\cref{fig:barplot_clustered} evaluates the impact of inexact least-squares solvers. Each bar plot (mean $\pm$ std) compares two settings: the first employs Cholesky decomposition, while the second uses MATLAB's \texttt{PCG} solver with a fixed iteration cap of $m = 30$ across all algorithms. Each experiment is repeated 100 times with randomly initialized $\boldsymbol{x}_0$.

Split-Merge consistently outperforms all baselines, and its runtime is further reduced by factors of $17\times$ (Lapla3) and $11\times$ (Lapla4) when switching from Cholesky to PCG. Among all methods, Lanczos exhibits the highest runtime variance, reflecting its instability. This phenomenon is further explained in \cref{fig:sin_time_Lapla}, where the convergence curve of $\sin(\theta_k)$ flattens prematurely, indicating possible stagnation. As noted in prior work \cite{parlett1998symmetric}, the orthogonality of the Lanczos basis matrix may progressively deteriorate during the iteration process due to accumulated floating-point rounding errors. Such behavior poses risks in high-accuracy recovery tasks, where convergence may fail to occur.

\section{Conclusions and Future Work}\label{sec:conclusion}

This work revisits the GEP through a difference-based formulation with a structured quadratic objective.
The objective exhibits bounded positive curvature and contains no non-strict saddle points, enabling global convergence of first-order methods without the need for line search.
By introducing a transform-domain perspective, we reveal the first-order nature of the power method and propose an accelerated PMD algorithm that allows for larger stepsizes and improved convergence.
We further extend the Split-Merge algorithm to the general GEP setting and uncover its theoretical connection to the standard EP through the transform-domain framework.
Empirical evaluations on both synthetic and real-world datasets demonstrate substantial improvements in computational efficiency and numerical stability over existing baselines.
Future work includes developing adaptive preconditioner strategies and extending the framework to broader classes of structured eigenvalue problems.

\appendix

\section{From $1$-GEP to $k$-GEP: A Recursive Strategy} \label{appendix:deflation}

In this section, we consider extending the Split-Merge algorithm to compute the top-$k$ leading generalized eigenvectors of the matrix pair $(\boldsymbol{A}, \boldsymbol{B})$. As suggested in \cite{allen2017doubly}, the $k$-GEP can be addressed by recursively applying a $1$-GEP solver as a meta-algorithm for $k$ iterations. In other words, instead of computing the top-$k$ generalized eigenspace simultaneously, we sequentially compute the top-$k$ eigenvectors one by one.

Starting with $\boldsymbol{A}_0 = \boldsymbol{A}$, the Split-Merge algorithm computes, at the $s$-th iteration, the leading generalized eigenvector of the matrix pair $(\boldsymbol{A}_{s-1}, \boldsymbol{B})$, yielding an approximate solution $\boldsymbol{u}_s$. Then, a deflation step is performed by updating
$$
\boldsymbol{A}_{s} \leftarrow \boldsymbol{P} \boldsymbol{A}_{s-1} \boldsymbol{P}^T,
$$
where the projection matrix is given by $\boldsymbol{P} = \boldsymbol{I} - \boldsymbol{B}\boldsymbol{u}_s\boldsymbol{u}_s^T$.

While the underlying idea of this recursive strategy is conceptually simple, its theoretical analysis requires certain algebraic techniques. For further details, we refer the reader to the overview section in \cite{allen2016lazysvd}.

\section{Proof of \cref{lemma:positive_lipschitz}} \label{appendix:proof_positive_lipschitz}

\begin{proof}
	Since $f : \mathbb{R}^n \to \mathbb{R}$ is continuously differentiable with Hessian bounded above by $L_+ \boldsymbol{I}$, we consider the second-order Taylor expansion:
	$$
	f(\boldsymbol{y}) = f(\boldsymbol{x}) + \nabla f(\boldsymbol{x})^T (\boldsymbol{y} - \boldsymbol{x}) + \int_0^1 (1 - t)(\boldsymbol{y} - \boldsymbol{x})^T \nabla^2 f(\boldsymbol{x} + t(\boldsymbol{y} - \boldsymbol{x})) (\boldsymbol{y} - \boldsymbol{x}) dt.
	$$
	Using the assumption $\nabla^2 f(\boldsymbol{x}) \preceq L_+ \boldsymbol{I}$ for all $\boldsymbol{x}$, we obtain:
	$$
	(\boldsymbol{y} - \boldsymbol{x})^T \nabla^2 f(\boldsymbol{x} + t(\boldsymbol{y} - \boldsymbol{x})) (\boldsymbol{y} - \boldsymbol{x}) \leq L_+ \| \boldsymbol{y} - \boldsymbol{x} \|^2.
	$$
	Therefore,
	$$
	\begin{aligned}
		&\int_0^1 (1 - t) (\boldsymbol{y} - \boldsymbol{x})^T \nabla^2 f(\boldsymbol{x} + t(\boldsymbol{y} - \boldsymbol{x})) (\boldsymbol{y} - \boldsymbol{x}) dt \\ 
		&\leq \int_0^1 (1 - t) L_+ \| \boldsymbol{y} - \boldsymbol{x} \|^2 \, dt = \frac{L_+}{2} \| \boldsymbol{y} - \boldsymbol{x} \|^2.
	\end{aligned}
	$$
	Substituting yields the claimed inequality \cref{eq:upper_lip}.
\end{proof}

\section{Proof of \cref{lemma:descent}} \label{appendix:proof_descent}

\begin{proof}
	Applying the inequality \cref{eq:upper_lip} with $\boldsymbol{y} = \bar{\boldsymbol{x}}$, we obtain
	$$
	\begin{aligned}
		f(\bar{\boldsymbol{x}}) &\leq f(\boldsymbol{x}) + \nabla f(\boldsymbol{x})^T (\bar{\boldsymbol{x}} - \boldsymbol{x}) + \frac{L_+}{2} \| \bar{\boldsymbol{x}} - \boldsymbol{x} \|^2 \\
		&= f(\boldsymbol{x}) - \alpha \| \nabla f(\boldsymbol{x}) \|^2 + \frac{L_+}{2} \alpha^2 \| \nabla f(\boldsymbol{x}) \|^2 \\
		&= f(\boldsymbol{x}) - \alpha \left( 1 - \frac{\alpha L_+}{2} \right) \| \nabla f(\boldsymbol{x}) \|^2.
	\end{aligned}
	$$
	Since the condition $\alpha \in (0, 2/L_+)$ ensures that $1 - \frac{\alpha L_+}{2} > 0$, it follows that
	$$
	f(\bar{\boldsymbol{x}}) \leq f(\boldsymbol{x}).
	$$
	This proves that the function value is non-increasing along the gradient descent trajectory.
\end{proof}

\section{Proof of \cref{thm:convergence_stationary_point}} \label{appendix:proof_convergence_stationary_point}

\begin{proof}
	By \cref{lemma:descent}, we have the following inequality for each iteration:
	$$
	f(\boldsymbol{x}_{k+1}) \leq f(\boldsymbol{x}_k) - \alpha \left( 1 - \frac{\alpha L_+}{2} \right) \| \nabla f(\boldsymbol{x}_k) \|^2.
	$$
	Define the constant $c := \alpha \left( 1 - \frac{\alpha L_+}{2} \right)$, which is strictly positive by assumption on the stepsize. The above inequality implies that:
	$$
	f(\boldsymbol{x}_{k}) - f(\boldsymbol{x}_{k+1}) \geq c \| \nabla f(\boldsymbol{x}_k) \|^2.
	$$
	Summing both sides from $k=0$ to $N-1$ yields
	$$
	f(\boldsymbol{x}_{0}) - f(\boldsymbol{x}_{N}) \geq c \sum_{k=0}^{N-1} \| \nabla f(\boldsymbol{x}_k) \|^2.
	$$
	Since $f$ is bounded below, the sequence $\left\{f(\boldsymbol{x}_k)\right\}$ is non-increasing and convergent. Hence, letting $N \rightarrow \infty$, we obtain:
	$$
	\sum_{k=0}^{\infty} \| \nabla f(\boldsymbol{x}_k) \|^2 \leq \infty.
	$$
	This implies that $\| \nabla f(\boldsymbol{x}_k) \| \rightarrow 0$ as $k \rightarrow \infty$, completing the proof.
\end{proof}

\section{Synthetic Dataset Generation} \label{appendix:results_condition_number}

We construct a synthetic matrix pair $(\boldsymbol{A}, \boldsymbol{B})$, where the eigenvalues of $\boldsymbol{A}$ and $\boldsymbol{B}$ are evenly spaced within the intervals $[1/\kappa_{\boldsymbol{A}},1]$ and $[1/\kappa_{\boldsymbol{B}},1]$, respectively. We fix $\kappa_{\boldsymbol{A}}=100$, while varying $\kappa_{\boldsymbol{B}} \in \{3, 5, 8, 10, 13, 30, 40, 50, 80, 100\}$ and the matrix dimension $n\in \{256, 512, 1024\}$.

\section*{Acknowledgments}
We would like to thank Prof. Mengmeng Song for her helpful discussions and valuable suggestions.

\bibliographystyle{siamplain}
\bibliography{references}
\end{document}

%% file: ex_shared.tex
\usepackage{lipsum}
\usepackage{amsfonts}
\usepackage{graphicx}
\usepackage{epstopdf}
\usepackage{algorithmic}
\ifpdf
  \DeclareGraphicsExtensions{.eps,.pdf,.png,.jpg}
\else
  \DeclareGraphicsExtensions{.eps}
\fi

\newsiamremark{remark}{Remark}
\newsiamremark{hypothesis}{Hypothesis}
\crefname{hypothesis}{Hypothesis}{Hypotheses}
\newsiamthm{claim}{Claim}
\newsiamremark{fact}{Fact}
\crefname{fact}{Fact}{Facts}

\headers{Split-Merge for Generalized Eigenvalue Problems}{Xiaozhi Liu and Yong Xia}

\title{Split-Merge Revisited: A Scalable Approach to Generalized Eigenvalue Problems\thanks{\funding{This work was funded by National Key Research and Development Program of China under grant 2021YFA1003303 and National Natural Science Foundation of China under grant 12171021.}}}

\author{Xiaozhi Liu\thanks{School of Mathematical Sciences, Beihang University, Beijing, 100191, People's Republic of China. (\email{xzliu@buaa.edu.cn}).}
\and Yong Xia \thanks{Corresponding author. School of Mathematical Sciences, Beihang University, Beijing, 100191, People's Republic of China. (\email{yxia@buaa.edu.cn}).}}

\usepackage{amsopn}

\usepackage{color}
\newcommand{\red}[1]{\textcolor{red}{#1}}

\usepackage{colortbl}
\definecolor{bblue}{rgb}{0.855,0.933,0.98}

\definecolor{rred}{HTML}{DC143C}

\usepackage{multirow}

\usepackage{makecell}

\usepackage{booktabs} 

\usepackage{subcaption} 

\usepackage{rotating}

\usepackage[normalem]{ulem}

\usepackage{pifont}
\def\cha{\ding{55}}
